\documentclass[12pt, reqno]{amsart}
\usepackage{amsmath, amsthm, amscd, amsfonts, amssymb, graphicx, color}
\usepackage[bookmarksnumbered, colorlinks, plainpages]{hyperref}
\hypersetup{colorlinks=true,linkcolor=red, anchorcolor=green, citecolor=cyan, urlcolor=red, filecolor=magenta, pdftoolbar=true}

\DeclareMathSymbol{\subsetneqq}{\mathbin}{AMSb}{36}

\newcommand{\R}{\mathbb{R}}

\newcommand{\C}{\mathbb{C}}

\newcommand{\beq}{\begin{eqnarray}}
\newcommand{\eeq}{\end{eqnarray}}
\newcommand{\bq}{\begin{equation}}
\newcommand{\eq}{\end{equation}}
\newcommand{\beqn}{\begin{eqnarray*}}
\newcommand{\eeqn}{\end{eqnarray*}}
\newcommand{\bex}{\begin{exo}}
\newcommand{\eex}{\end{exo}}
\newcommand{\ben}{\begin{enumerate}}
\newcommand{\een}{\end{enumerate}}

\newtheorem{th1}{{\bf Theorem}}[section]
\newtheorem{thm}[th1]{{\bf Theorem}}
\newtheorem{lem}[th1]{{\bf Lemma}}
\newtheorem{prop}[th1]{{\bf Proposition}}

\newtheorem{rems}[th1]{\bf Remarks}
\newtheorem{defi}[th1]{\bf Definition}

\author[T. Saanouni and H. Hezzi]{Tarek Saanouni and Hanene Hezzi}
\address{Department of Mathematics, College of Sciences and Arts of Uglat Asugour, Qassim University, Buraydah, Kingdom of Saudi Arabia.}
\address{University of Tunis El Manar, Faculty of Science of Tunis, LR03ES04 partial differential Equations and applications, 2092 Tunis, Tunisia.}
\email{\sl t.saanouni@qu.edu.sa}
\email{\sl tarek.saanouni@ipeiem.rnu.tn}
\email{\sl hezzi$_{-}{82}$hanen@yahoo.fr}

\subjclass[2010]{35Q55}
\keywords{Fourth-order Schr\"odinger equation, inhomogeneous, non-radial, scattering.}

\title[IBNLS]{Scattering for a class of non-radial inhomogeneous bi-harmonic Hartree equations}

\date{\today}
\begin{document}
\begin{abstract}
This manuscript proves the energy scattering of global solutions to a repulsive fourth-order generalized Hartree equation with non-radial data in the inter-critical regime. This work uses a new approach due to Dodson-Murphy \cite{dm} and extends the previous work \cite{st4} by removing the spherically symmetric assumption on the data.
\end{abstract}
\maketitle
\vspace{ 1\baselineskip}
\renewcommand{\theequation}{\thesection.\arabic{equation}}
\section{Introduction}
This work treats the energy scattering theory for the following inhomogeneous focusing Choquard equation
\begin{equation}
\left\{
\begin{array}{ll}
i\dot u+\Delta^2  u-(I_\alpha*|\cdot|^b|u|^p)|x|^{b}|u|^{p-2}u=0 ;\\
u(0,.)=u_0.
\label{S}
\end{array}
\right.
\end{equation}

Here and hereafter $u:\R\times\R^N \to \C$, for a natural integer $N\geq5$. The singular inhomogeneous term is $|\cdot|^b$, for a certain $b<0$. The Riesz-potential is defined on $\R^N$ by 
$$I_\alpha:x\mapsto\frac{\Gamma(\frac{N-\alpha}2)}{\Gamma(\frac\alpha2)\pi^\frac{N}22^\alpha|x|^{N-\alpha}},\quad  0<\alpha<N.$$ 

The limiting case $b=0$ corresponds to the homogeneous fourth-order Schr\"odinger problem considered first in \cite{Karpman,Karpman 1} to take into account the role of small fourth-order dispersion terms in the propagation of intense laser beams in a bulk medium with a Kerr non-linearity.\\

The equation \eqref{S} is invariant under the scaling 
$$u_\lambda=\lambda^\frac{4+2b+\alpha}{2(p-1)}u(\lambda^4\,\cdot,\lambda\,\cdot),\quad\lambda>0.$$
The homogeneous Sobolev norm gives the critical Sobolev index denoted by $s_c$ as follows
$$\|u_\lambda\|_{\dot H^s}=\lambda^{s-(\frac N2-\frac{4+2b+\alpha}{2(p-1)})}\|u(\lambda^4\,\cdot)\|_{\dot H^s}:=\lambda^{s-s_c}\|u(\lambda^4\,\cdot)\|_{\dot H^s},$$
In this note, one considers the inter-critical regime $0<s_c<2$, which corresponds to the mass super-critical and energy sub-critical case.\\

Let us recall give a brief literature about the inhomogeneous bi-harmonic non-linear Schr\"odinger equation (IBNLS). The finite time blow-up of solutions to the IBNLS for negative energy with a source term $|x|^{-2}|u|^{\frac4N}u$ was considered in \cite{cow}. The local well-posedness in the energy space was treated in \cite{gp}. See also \cite{lz}, for the well-posedness in $H^s$, $0<s\leq2$ and \cite{cgp} for existence of global and non-global solutions in $\dot H^{s_c}\cap \dot H^2$. The case of a non-local source term was considered by the author in \cite{st2}. The scattering of global spherically symmetric solutions under the ground state threshold was established recently by the author \cite{st4}. Moreover, the scattering without any radial assumption for the IBNLS with a local source term was proved very recently \cite{lcg}.\\

It is the aim of this note to extend \cite{st4} to the non-radial regime and \cite{lcg} to the generalized Hartree equation. The challenge of this work is to deal with the non-local source term with use of a Hardy-Littlewood-Sobolev estimate. The main ingredient here is the use of the decay of the inhomogeneous term $|\cdot|^b$ instead of the spherically symmetric assumption. So, the scattering in the limiting case $b=0$ is not a consequence of this manuscript. \\   
 
 The rest of the note is organized as follows. In section 2, one gives the main result and some useful estimates. Section three contains a proof of a Morawetz identity. In section four one proves a scattering criterion. The last section proves the main Theorem.\\

Here and hereafter, $C$ denotes a constant which may vary from line to another. Denote the Lebesgue space $L^r:=L^r({\R^N})$ with the usual norm $\|\cdot\|_r:=\|\cdot\|_{L^r}$ and $\|\cdot\|:=\|\cdot\|_2$. The inhomogeneous Sobolev space $H^2:=H^2({\R^N})$ is endowed with the norm 
$$ \|\cdot\|_{H^2} := \Big(\|\cdot\|^2 + \|\Delta\cdot\|^2\Big)^\frac12.$$
Let us denote also $C_T(X):=C([0,T],X)$ and $X_{\bf r}$ the set of radial elements in $X$. Moreover, for an eventual solution to \eqref{S}, $T^*>0$ denotes it's lifespan. Finally, $x^+$ is a real numbers near to $x$ satisfying $x^+>x$.

\section{Background and main results}
This section contains the contribution of this paper and some standard estimates needed in the sequel.
\subsection{Preliminary}
Take for $R>0$, $\psi_R:=\psi(\frac\cdot R)$, where $0\leq\psi\leq1$ is a radial smooth function satisfying
$$\psi\in C_0^\infty(\R^N),\quad supp(\psi)\subset \{|x|<1\}, \quad\psi=1\,\,\mbox{on}\,\,\{|x|<\frac12\}.$$
The mass-critical and energy-critical exponents for the Choquard problem \eqref{S} are
$$p_*:=1+\frac{\alpha+4+2b}N\quad\mbox{and}\quad p^*=1+\frac{4+2b+\alpha}{N-4}.$$
Solutions of the Choquard problem \eqref{S} satisfy the conservation of the mass and the energy
\begin{align*}
M[u]&:=\|u\|^2;\\
E[u]&:=\|\Delta u\|^2-\frac1p\int_{\R^N}(I_\alpha*|\cdot|^b|u|^p)|x|^b|u|^p\,dx.
\end{align*}
Let $\phi$ a ground state solution to the the elliptic problem
$$\phi+\Delta^2\phi-(I_\alpha*|\cdot|^b|\phi|^p)|x|^b|\phi|^{p-2}\phi=0,\quad0\neq\phi\in H^2,$$
and $u$ a solution to \eqref{S}. The following scale invariant quantities describe the dichotomy of global/non-global existence of solutions \cite{guo}.
\begin{gather*}
\mathcal{ME}[u]:=\Big(\frac{E[u]}{E[\phi]}\Big)\Big(\frac{M[u]}{M[\phi]}\Big)^\frac{2-s_c}{s_c};\\
\mathcal{MG}[u]:=\Big(\frac{\|\Delta u\|}{\|\Delta\phi\|}\Big)\Big(\frac{\|u\|}{\|\phi\|}\Big)^\frac{2-s_c}{s_c}.
\end{gather*}
The local well-posedness of the above problem \eqref{S} in the energy space for $2\leq p<p^*$ was proved in \cite{st2} under the assumptions denoted for simplicity $(N,\alpha,b)$ satisfies $(\mathcal C)$ if, $0<\alpha<N$ and $\max\{-(N+\alpha),-4(1+\frac\alpha N),N-8-\alpha\}<2b<0$ and $[N\geq5$ or $3\leq N\leq 4$ and $2\alpha+4b+N>0]$. Moreover, the global existence versus finite time blow-up of energy solutions under the ground state threshold was obtained in \cite{st4}. Precisely,
\begin{thm}\label{sctr1}
Let  $(N,\alpha,b)$ satisfying $(\mathcal C)$ and $\max\{p_*,x_\alpha\}<p<p^*$ such that $p\geq\max\{2,\frac32+\frac\alpha N\}$. Let $u_0\in H^2_{\bf r}$ satisfying
\begin{equation}\label{ss}
\max\Big\{\mathcal{ME}[u_0],\mathcal{MG}[u_0]\Big\}<1.
\end{equation}
 Take $u\in C_{T^*}(H^2_{\bf r})$ be a maximal solution to \eqref{S}. Then, $u$ is global and there exists $u_\pm\in H^2$ such that
$$\lim_{t\to\pm\infty}\|u(t)-e^{it\Delta^2}u_\pm\|_{H^2}=0.$$
\end{thm}
\begin{rems}
Note that
\begin{enumerate}
\item[1.]
$x_\alpha$ is the positive root of the polynomial
$$(X-1)(2X-1)-\frac{4+2b+\alpha}{N-4};$$
\item[2.]
the proof follows a new approach due to \cite{dm};
\item[3.]
the spherically symmetric assumption is essential in the used method.
\end{enumerate}
\end{rems}
The following variational estimates \cite{st4} are needed in the proof of the scattering of global solutions to the focusing Choquard problem \eqref{S}. 
\begin{lem}\label{bnd}
Take $(N,\alpha,b)$ satisfying $(\mathcal C)$ and $p_*<p<p^*$ such that $p\geq2$. Let $u_0\in H^2$ satisfying \eqref{ss} and $u\in C(\R,H^2)$ be the solution to \eqref{S}. Then, 
\begin{enumerate}
\item[1.]
there exists $0<\delta<1$ such that
$$\max\Big\{\sup_{t\in\R} \mathcal{ME}[u(t)],\sup_{t\in\R}\mathcal{MG}[u(t)]\Big\}<1-\delta.$$
\item[2.]
There exists $R_0:=R_0(\delta,M(u),\phi)>0$ such that for any $R>R_0$,
$$\sup_{t\in\R}\|\psi_R u(t)\|^{2-s_c}\|\Delta(\psi_Ru(t))\|^{s_c}<(1-\delta)\|\phi\|^{2-s_c}\|\Delta\phi\|^{s_c}.$$
Moreover, there exists $\delta'>0$ such that
$$\|\Delta(\psi_Ru)\|^2-\frac B{2p}\int_{\R^N}(I_\alpha*|\cdot|^b|\psi_Ru|^p)|x|^b|\psi_Ru|^p\,dx\geq\delta'\|\psi_Ru\|_{\frac{2Np}{N+\alpha+2b}}^{2p}.$$ 
\end{enumerate}
\end{lem}
\subsection{Main result}
The main goal of this manuscript is to prove the following scattering result.
\begin{thm}\label{sctr}
Let  $(N,\alpha,b)$ satisfying $(\mathcal C)$ and $p_*<p<p^*$ such that $p\geq2$ and $N\geq5$. Let $u_0\in H^2$ satisfying \eqref{ss} and $u\in C_{T^*}(H^2)$ be the maximal solution to \eqref{S}. Then, $u$ is global and there exists $u_\pm\in H^2$ such that
$$\lim_{t\to\pm\infty}\|u(t)-e^{it\Delta^2}u_\pm\|_{H^2}=0.$$
\end{thm}
\begin{rems}
Note that
\begin{enumerate}
\item[1.]
the global existence of solutions under the assumption \eqref{ss} was proved in \cite{st2};
\item[2.]
the proof follows a new approach due to \cite{dm} and avoids the concentration-compactness method introduced by \cite{km};
\item[3.]
the main novelty here is the removal of the spherically symmetric of the data;
\item[4.]
one exploits the decay of the inhomogeneous term $|\cdot|^b$ of the source term, instead of the spherically symmetric assumption;
\item[5.]
the condition $N\geq5$ is used in the scattering criteria;
\item[6.]
it seems that the scattering for non-radial data in lower dimensions still remains open;
\item[7.]
in the limiting case $b=0$, it seems that the scattering for non-radial data is still open.
\end{enumerate}
\end{rems}
\subsection{Useful estimates}
Let us gather some classical tools needed in the sequel.
\begin{defi}\label{adm}
Take $N\geq1$ and $s\in[0,2)$. A couple of real numbers $(q,r)$ is said to be $s$-admissible $($admissible for $0$-admissible$)$ if 
$$\frac{2N}{N-2s}\leq r<\frac{2N}{N-4},\quad2\leq q,r\leq\infty\quad\mbox{and}\quad N(\frac12-\frac1r)=\frac4q+s.$$
Denote the set of $s$-admissible pairs by $\Gamma_s$ and $\Gamma:=\Gamma_0$. If $I$ is a time slab, one denotes the Strichartz spaces
$$S^s(I):=\cap_{(q,r)\in\Gamma_s}L^q(I,L^r).$$
\end{defi}
Recall the Strichartz estimates \cite{bp,guo}.
\begin{prop}\label{prop2}
Let $N \geq 1$, $0\leq s<2$ and $t_0\in I\subset \R$, an interval. Then,
\begin{enumerate}
\item[1.]
$\sup_{(q,r)\in\Gamma}\|u\|_{L^q(I,L^r)}\lesssim\|u(t_0)\|+\inf_{(\tilde q,\tilde r)\in\Gamma}\|i\dot u+\Delta^2 u\|_{L^{\tilde q'}(I,L^{\tilde r'})}$;
\item[2.]
$\sup_{(q,r)\in\Gamma}\|\Delta u\|_{L^q(I,L^r)}\lesssim\|\Delta u(t_0)\|+\|i\dot u+\Delta^2 u\|_{L^2(I,\dot W^{1,\frac{2N}{2+N}})}, \quad\forall N\geq3$;
\item[3.]
$\sup_{(q,r)\in\Gamma_s}\|u\|_{L^q(I,L^r)}\lesssim\|u(t_0)\|_{\dot H^s}+\inf_{(\tilde q,\tilde r)\in\Gamma_{-s}}\|i\dot u+\Delta^2 u\|_{L^{\tilde q'}(I,L^{\tilde r'})}$.
\end{enumerate}
\end{prop}
Let us recall a Hardy-Littlewood-Sobolev inequality \cite{el,sw}.
\begin{lem}\label{hls}
Take $N\geq1$.
\begin{enumerate}
\item[1.]
Let $0 <\lambda < N$ and $1<r,s<\infty$ satisfying $2=\frac1r +\frac1s +\frac\lambda N $. Thus,
$$\int_{\R^N\times\R^N} \frac{u(x)v(y)}{|x-y|^\lambda}\,dx\,dy\leq C_{N,s,\lambda}\|u\|_{r}\|v\|_{s},\quad\forall u\in L^r,\,\forall v\in L^s.$$
\item[2.]
Let $0 <\alpha < N$ and $1<r,s,q<\infty$ satisfying $1+\frac\alpha N=\frac1q+\frac1r+\frac1s$. Thus,
$$\|(I_\alpha*u)v\|_{r'}\leq C_{N,s,\alpha}\|u\|_{s}\|v\|_{q},\quad\forall u\in L^s, \,\forall v\in L^q.;$$
\item[3.]
Let $0 <\alpha < N$ and $1<r,s,q<\infty$ satisfying $1+\frac{\alpha-\gamma-\mu} N=\frac1q+\frac1r+\frac1s$ and $0<-\gamma<\frac N{s'}$, $0<-\mu<\frac N{q'}$. Thus,
$$\|(I_\alpha*|\cdot|^\gamma u)|\cdot|^\mu v\|_{r'}\leq C_{N,s,\alpha,\gamma,\mu}\|u\|_{s}\|v\|_{q},\quad\forall u\in L^s, \,\forall v\in L^q.;$$
\end{enumerate}
\end{lem}
Finally, let us give an abstract result \cite{st4}.
\begin{lem}\label{abs}
Let $T>0$ and $X\in C([0,T],\R_+)$ such that $$X\leq a+bX^{\theta}\mbox{ on } [0,T],$$
where $a$, $b>0$, $\theta>1$, $a<(1-\frac{1}{\theta})(\theta b)^{\frac{1}{1-\theta}}$ and $X(0)\leq (\theta b)^{\frac{1}{1-\theta}}$. Then
$$X\leq\frac{\theta}{\theta -1}a \mbox{ on } [0,T].$$
\end{lem}
\section{Morawetz identity}
One adopts the convention that repeated indexes are summed. 
Recall a classical Morawetz estimate \cite{st4} satisfied by the energy global solutions to the inhomogeneous Choquard problem \eqref{S}. 
\begin{lem}\label{mrwtz}
Take $(N,\alpha,b)$ satisfying $(\mathcal C)$, $p_*< p<p^*$ such that $p\geq2$ and $u\in C_{T}(H^2)$ be a local solution to \eqref{S}. Let $a:\R^N\to\R$ be a convex smooth function  and the real function defined on $[0,T)$, by
$$M_a:t\to2\int_{\R^N}\nabla a(x)\cdot\Im\Big(\nabla u(t,x)\bar u(t,x)\Big)\,dx.$$
Then, the following equality holds on $[0,T)$,
\begin{eqnarray*}
M'_a
&=&2\int_{\R^N}\Big(2\partial_{jk}\Delta a\partial_ju\partial_k\bar u-\frac12(\Delta^3a)|u|^2-4\partial_{jk}a\partial_{ik}u\partial_{ij}\bar u+\Delta^2a|\nabla u|^2\Big)\\
&+&2\Big((1-\frac2p)\int_{\R^N}\Delta a(I_\alpha*|\cdot|^b|u|^{p})|x|^b|u|^p\,dx-\frac2{p}\int_{\R^N}\partial_ka\partial_k(|x|^b[I_\alpha*|\cdot|^b|u|^p])|u|^{p}\,dx\Big).
\end{eqnarray*}
\end{lem}
Let us write the main result of this section.
\begin{prop}\label{bnd1'}
Take $(N,\alpha,b)$ satisfying $(\mathcal C)$ and $p_*<p<p^*$ such that $p\geq2$. Let $u_0\in H^2$ satisfying \eqref{ss}.
Then, for any $T>0$, one has
$$\int_0^T\|u(t)\|_{L^\frac{2Np}{N+\alpha+2b}(|x|<R)}^{2p}\,dt\leq CT^{1/(1-b)}.$$
\end{prop}
\begin{proof}
Take a smooth real function such that 
$$f:r\to\left\{
\begin{array}{ll}
\frac{r^2}2,\quad\mbox{if}\quad 0\leq r\leq1;\\
r,\quad\mbox{if}\quad  r\geq2,
\end{array}
\right.
$$
moreover,
$$\min\{f',f''\}\geq0,\quad\mbox{on}\quad [1,2].$$
 Note that for $|x|\leq 1$, confusing for $x\in\R^N$, $f(x):=f(|x|)$, one has
$$f_{ij} =\delta_{ij},\quad\Delta f =N\quad\mbox{and}\quad \partial^\gamma f= 0 \quad\mbox{for}\quad |\gamma| \geq 3.$$
Finally, one denotes for $R>0$, the smooth radial function defined on $\R^N$ by $f_R:=R^2f(\frac{|\cdot|}R)$ and the real function $M_R:=M_{f_R}$. 
Using the estimate $\|\nabla^\gamma f_R\|_\infty\lesssim R^{2-|\gamma|}$, one has
\begin{align*}
|\int_{\R^N}\Delta^2f_R|\nabla u|^2\,dx|+|\int_{\R^N}\partial_{jk}\Delta f_R\partial_ju\partial_k\bar u\,dx|&\lesssim R^{-2};\\
|\int_{\R^N}(\Delta^3f_R)|u|^2\,dx|&\lesssim R^{-4}.
\end{align*}
Thus, by Morawetz estimate in Lemma \ref{mrwtz}, one gets
\begin{eqnarray*}
M_R'
&=&-\frac4{p}\int_{\R^N}\partial_kf_R\partial_k[(I_\alpha*|\cdot|^b|u|^p)|x|^b]|u|^{p}\,dx+O(R^{-2})\\
&+&2\Big({N}(1-\frac2p)\int_{\{|x|<R\}}(I_\alpha*|\cdot|^b|u|^p)|x|^b|u|^p\,dx-4\int_{\{|x|<R\}}|\Delta u|^2\,dx\Big)\\
&+&2\Big((1-\frac2p)\int_{\{|x|>R\}}\Delta f_R(I_\alpha*|\cdot|^b|u|^p)|x|^b|u|^p\,dx-4\int_{\{|x|>R\}}\partial_{jk}f_R\partial_{ik}u\partial_{ij}\bar u\,dx\Big).
\end{eqnarray*}
Moreover, denoting the radial derivative by $\partial_r:=\nabla\cdot\frac{x}{|x|}$, one writes
$$\partial_{jk}f_R=\Big(\frac{\delta_{jk}}r-\frac{x_jx_k}{r^3}\Big)\partial_rf_R+\frac{x_jx_k}{r^2}\partial_r^2f_R.$$
So,
\begin{eqnarray*}
\int_{\{|x|>R\}}\partial_{jk}f_R\partial_{ik}u\partial_{ij}\bar u\,dx
&=&\int_{\{|x|>R\}}\Big[\Big(\frac{\delta_{jk}}r-\frac{x_jx_k}{r^3}\Big)\partial_rf_R+\frac{x_jx_k}{r^2}\partial_r^2f_R\Big]\partial_{ik}u\partial_{ij}\bar u\,dx\\
&=&\sum_{i=1}^N\int_{\{|x|>R\}}\Big(|\nabla u_i|^2-\frac{|x\cdot\nabla u_i|^2}{|x|^2}\Big)\frac{\partial_rf_R}{|x|}\,dx\\
&+&\sum_{i=1}^N\int_{\{|x|>R\}}\frac{|x\cdot\nabla u_i|^2}{|x|^2}\partial_r^2f_R\,dx\\
&=&\sum_{i=1}^N\int_{\{|x|>R\}}|\not\nabla u_i|^2\frac{\partial_rf_R}{|x|}\,dx+\sum_{i=1}^N\int_{\{|x|>\frac R2\}}\frac{|x\cdot\nabla u_i|^2}{|x|^2}\partial_r^2f_R\,dx,
\end{eqnarray*}
where the angular gradient is $\not\nabla :=\nabla -\frac{x\cdot\nabla}{|x|^2}x$. Then,
\begin{eqnarray*}
-M_R'
&\geq&\frac4{p}\int_{\R^N}\partial_kf_R\partial_k[(I_\alpha*|\cdot|^b|u|^p)|x|^b]|u|^{p}\,dx+O(R^{-2})\\
&+&2\Big(4\int_{\{|x|<R\}}|\Delta u|^2\,dx-{N}(1-\frac2p)\int_{\{|x|<R\}}(I_\alpha*|\cdot|^b|u|^p)|x|^b|u|^p\,dx\Big)\\
&-&(1-\frac2p)\int_{\{|x|>R\}}\Delta f_R(I_\alpha*|\cdot|^b|u|^p)|x|^b|u|^p\,dx.
\end{eqnarray*}
Take the quantity 
\begin{eqnarray*}
(A)
&:=&\int_{\R^N}\partial_kf_R\partial_k[(I_\alpha*|\cdot|^b|u|^p)|x|^b]|u|^{p}\,dx\\
&=&-(N-\alpha)\int_{\R^N}\nabla f_R\Big(\frac\cdot{|x|^2}I_\alpha*|\cdot|^b|u|^p\Big)|x|^b|u|^{p}\,dx+b\int_{\R^N}\frac{\nabla f_R\cdot x}{|x|^2}(I_\alpha*|\cdot|^b|u|^p)|x|^b|u|^p\,dx\\
&:=&-(N-\alpha)(I)+b(II).
\end{eqnarray*}
With the properties of $f_R$, one writes
\begin{eqnarray*}
(II)
&=&\int_{|x|<R}(I_\alpha*|\cdot|^b|u|^p)|x|^b|u|^{p}\,dx+O\Big(\int_{|x|>R}(I_\alpha*|\cdot|^b|u|^p)|x|^b|u|^p\,dx\Big).
\end{eqnarray*}
With calculus done in \cite{st4}, one has
\begin{eqnarray*}
(I)
&=&\frac12\int_{|x|<R}\int_{|y|<R}I_\alpha(x-y)|y|^b|u(y)|^p|x|^b|u(x)|^{p}\,dx+O\Big(\int_{|x|>R}(I_\alpha*|\cdot|^b|u|^p)|x|^b|u|^p\,dx\Big)\\
&=&\frac12\int_{|x|<R}(I_\alpha*|\cdot|^b|u|^p)|x|^b|u(x)|^{p}\,dx+O\Big(\int_{|x|>R}(I_\alpha*|\cdot|^b|u|^p)|x|^b|u|^p\,dx\Big).
\end{eqnarray*}
Thus, 
\begin{eqnarray*}
(A)
&=&(b-\frac{N-\alpha}2)\int_{|x|<R}(I_\alpha*|\cdot|^b|u|^p)|x|^b|u(x)|^{p}\,dx
+O\Big(\int_{|x|>R}(I_\alpha*|\cdot|^b|u|^p)|x|^b|u|^p\,dx\Big).
\end{eqnarray*}
So,
\begin{eqnarray*}
-M_R'
&\geq&2\Big(4\int_{\{|x|<R\}}|\Delta u|^2\,dx-{N}(1-\frac2p)\int_{\{|x|<R\}}(I_\alpha*|\cdot|^b|u|^p)|x|^b|u|^p\,dx\Big)\\
&+&\frac4{p}(A)+O(R^{-2})+O\Big(\int_{\{|x|>R\}}\Delta f_R(I_\alpha*|\cdot|^b|u|^p)|x|^b|u|^p\,dx\Big)\\
&\geq&2\Big(4\int_{\{|x|<R\}}|\Delta u|^2\,dx-{N}(1-\frac2p)\int_{\{|x|<R\}}(I_\alpha*|\cdot|^b|u|^p)|x|^b|u|^p\,dx\Big)\\
&+&\frac4{p}(b-\frac{N-\alpha}2)\int_{|x|<R}(I_\alpha*|\cdot|^b|u|^p)|x|^b|u(x)|^{p}\,dx\\
&+&O(R^{-2})+O\Big(\int_{\{|x|>R\}}(I_\alpha*|\cdot|^b|u|^p)|x|^b|u|^p\,dx\Big).
\end{eqnarray*}
Let us write
\begin{eqnarray*}
\int_{\{|x|>R\}}(I_\alpha*|\cdot|^b|u|^p)|x|^b|u|^p\,dx
&\leq&R^b\int_{\R^N}(I_\alpha*|\cdot|^b|u|^p)|u|^p\,dx\\
&\leq&R^b\int_{\R^N}\Big(\int_{\{|y|<1\}}+\int_{\{|y|>1\}}\Big)I_\alpha(x-y)|y|^b|u(y)|^p|u(x)|^p\,dy\,dx.
\end{eqnarray*}
Now, with Hardy-Littlewood-Sobolev estimate, one has
\begin{eqnarray*}
\int_{\R^N}\int_{\{|y|<1\}}I_\alpha(x-y)|y|^b|u(y)|^p|u(x)|^p\,dy\,dx
&\lesssim&\||\cdot|^b\|_{L^a(|x|<1)}\|u\|_r^{2p},
\end{eqnarray*}
where 
$$1+\frac\alpha N=\frac1a+\frac{2p}r.$$
Taking $a<\frac N{-b}$, one gets $1+\frac\alpha N-\frac{2p}r>\frac{-b}N$ and equivalently 
$$\frac{2Np}{N+\alpha+b}<r.$$
Since $p<p^*$, there is $r\in[2,\frac{2N}{N-4}]$ satisfying the above estimate. Thus,
\begin{eqnarray*}
\int_{\R^N}\int_{\{|y|<1\}}I_\alpha(x-y)|y|^b|u(y)|^p|u(x)|^p\,dy\,dx
&\lesssim&\|u\|_{H^2}^{2p}.
\end{eqnarray*}
Now, with Hardy-Littlewood-Sobolev estimate, one has
\begin{eqnarray*}
\int_{\R^N}\int_{\{|y|>1\}}I_\alpha(x-y)|y|^b|u(y)|^p|u(x)|^p\,dy\,dx
&\lesssim&\||\cdot|^b\|_{L^a(|x|>1)}\|u\|_r^{2p},
\end{eqnarray*}
where 
$$1+\frac\alpha N=\frac1a+\frac{2p}r.$$
Taking $a>\frac N{-b}$, one gets $1+\frac\alpha N-\frac{2p}r<\frac{-b}N$ and equivalently 
$$r<\frac{2Np}{N+\alpha+b}.$$
Since $p_*<p$, there is $r\in[2,\frac{2N}{N-4}]$ satisfying the above estimate. Thus,
\begin{eqnarray*}
\int_{\R^N}\int_{\{|y|>1\}}I_\alpha(x-y)|y|^b|u(y)|^p|u(x)|^p\,dy\,dx
&\lesssim&\|u\|_{H^2}^{2p}.
\end{eqnarray*}
Thus,
\begin{eqnarray*}
-M_R'
&\geq&8\Big(\int_{|x|<R}|\Delta u|^2\,dx-\frac B{2p}\int_{|x|<R}(I_\alpha*|\cdot|^b|u|^p)|x|^b|u|^p\,dx\Big)+O(R^{b}).
\end{eqnarray*}
On the other hand, an expansion via the properties of $\psi$ gives
\begin{eqnarray*}
\|\Delta(\psi_R u)\|^2
&\leq&\|\psi_R\Delta u\|^2+ C(u_0,\phi)R^{-2}\\
&\leq&\int_{|x|<R}|\Delta u|^2-\int_{\frac R2<|x|<R}(1-\psi_R^2)|\Delta u|^2\,dx+C(u_0,\phi)R^{-2}.
\end{eqnarray*}
Moreover,
\begin{eqnarray*}
&&\int_{\R^N}(I_\alpha*|\cdot|^b|\psi_Ru|^p)|x|^b|\psi_Ru|^p\,dx-\int_{\R^N}(I_\alpha*|\cdot|^b(1-\psi^p)|u|^p)|x|^b|u|^p\,dx\\
&=&\int_{\R^N}(I_\alpha*|\cdot|^b|u|^p)|x|^b|\psi_Ru|^p\,dx\\
&=&\int_{|x|<R}(I_\alpha*|\cdot|^b|u|^p)|x|^b|u|^p\,dx-\int_{\frac R2<|x|<R}(I_\alpha*|\cdot|^b|u|^p)|x|^b(1-\psi_R^p)|u|^p\,dx.
\end{eqnarray*}
Then,
$$
\int_{\R^N}(I_\alpha*|\cdot|^b|\psi_Ru|^p)|x|^b|\psi_Ru|^p\,dx=\int_{|x|<R}(I_\alpha*|\cdot|^b|u|^p)|x|^b|u|^p\,dx+O(R^b).
$$
So, with Lemma \ref{bnd},
one gets 
\begin{eqnarray*}
\sup_{[0,T]}|M|
&\geq&8\int_0^T\Big(\int_{\R^N}|\Delta(\psi_R u)|^2\,dx-\frac{B}{2p}\int_{\R^N}(I_\alpha*|\psi_Ru|^p)|\psi_Ru|^p\,dx\Big)\,dt+\mathcal O(R^{b})T\\
&\geq&8\delta'\int_0^T\|\psi_Ru(t)\|_{\frac{2Np}{N+\alpha+2b}}^{2p}\,dt+\mathcal O(R^{b})T\\
&\geq&8\delta'\int_0^T\|u(t)\|_{L^\frac{2Np}{N+\alpha+2b}(|x|<R)}^{2p}\,dt+\mathcal O(R^{b})T.
\end{eqnarray*}
The previous calculus gives
\begin{eqnarray*}
\int_0^T\|u(t)\|_{L^\frac{2Np}{N+\alpha+2b}(|x|<R)}^{2p}\,dt
&\leq& C\Big(\sup_{[0,T]}|M|+TR^{b}\Big)\\
&\leq& C\Big(R+TR^{b}\Big).
\end{eqnarray*}
Taking $R=T^{1/(1-b)}>>1$, one gets the requested estimate. 
For $0<T<<1$, the proof follows with Sobolev injections. 
\end{proof}
As a consequence, one has the following energy evacuation.
\begin{lem}\label{evac1}
Take $(N,\alpha,b)$ satisfying $(\mathcal C)$. Let $p_*<p<p^*$ such that $p\geq2$ and $u_0\in H^2$ satisfying \eqref{ss}.
Then, there exists a sequence of real numbers $t_n\to\infty$ such that the global solution to \eqref{S} satisfies
$$\lim_n\int_{\{|x|<R\}}|u(t_n,x)|^2\,dx=0,\quad\mbox{for all}\quad R>0.$$
\end{lem}
\begin{proof}
Take $t_n\to\infty$. By H\"older estimate
$$\int_{\{|x|<R\}}|u(t_n,x)|^2\,dx\leq R^{\frac{2B}p}\|u(t_n)\|_{L^\frac{2Np}{N+\alpha+2b}(|x|<R)}^2\to0.$$
Indeed, by the previous Lemma 
$$\|u(t_n)\|_{L^\frac{2Np}{N+\alpha+2b}(|x|<R)}\to0.$$
\end{proof}
\section{Scattering Criterion}
In this section one proves the next result. 
\begin{prop}\label{crt}
Take $(N,\alpha,b)$ satisfying $(\mathcal C)$. Let $p_*<p<p^*$ such that $p\geq2$. Let $u\in C(\R,H^2)$ be a global solution to \eqref{S}. Assume that 
$$0<\sup_{t\geq0}\|u(t)\|_{H^2}:=E<\infty.$$
There exist $R,\epsilon>0$ depending on $E,N,p,b,\alpha$ such that if
$$\liminf_{t\to+\infty}\int_{|x|<R}|u(t,x)|^2\,dx<\epsilon,$$
then, $u$ scatters for positive time. 
\end{prop}
\begin{proof}
By Lemma \ref{bnd}, $u$ is bounded in $H^2$. Take $\epsilon>0$ near to zero and $R(\epsilon)>>1$ to be fixed later. Let us give a technical result \cite{st4}.
\begin{lem}\label{tch}
Let $(N,b,\alpha)$ satisfying $(\mathcal C)$ and $p_*<p<p^*$ satisfying $p\geq2$. Then, there exists $\theta\in(0,2p-1)$ such that the global solution to \eqref{S} satisfies
$$\|u-e^{i.\Delta}u_0\|_{S^{s_c}(I)}\lesssim \|u\|_{L^\infty(I,H^{2})}^\theta\|u\|^{2p-1-\theta}_{L^a(I,L^r)},$$
for certain $(a,r)\in\Gamma_{s_c}$.
\end{lem}
The following result is the key to prove the scattering criterion.
\begin{prop}\label{fn1}
Take $(N,\alpha,b)$ satisfying $(\mathcal C)$. Let $p_*<p<p^*$ such that $p\geq2$. Let $u_0\in H^2$ satisfying \eqref{ss}.
Then, for any $\varepsilon>0$, there exist $T,\mu>0$ such that the global solution to \eqref{S} satisfies
$$\|e^{i(\cdot-T)\Delta^2}u(T)\|_{L^a((T,\infty),L^r)}\lesssim \varepsilon^\mu.$$
\end{prop}
\begin{proof}
Let $0<\beta<<1$ and $T>\varepsilon^{-\beta}>0$. By the integral formula
\begin{eqnarray*}
e^{i(\cdot-T)\Delta^2}u(T)
&=&e^{i\cdot\Delta^2}u_0+i\int_0^Te^{i(\cdot-s)\Delta^2}[(I_\alpha*|\cdot|^b|u|^p)|x|^b|u|^{p-2}u]\,ds\\
&=&e^{i\cdot\Delta^2}u_0+i\Big(\int_0^{T-\varepsilon^{-\beta}}+\int_{T-\varepsilon^{-\beta}}^T\Big)e^{i(\cdot-s)\Delta^2}\mathcal N\,ds\\
&:=&e^{i\cdot\Delta^2}u_0+F_1+F_2.
\end{eqnarray*}
Take the real numbers
\begin{gather*}
a:=\frac{2(2p-\theta)}{2-s_c},\quad d:=\frac{2(2p-\theta)}{2+(2p-1-\theta)s_c};\\
r:=\frac{2N(2p-\theta)}{(N-2s_c)(2p-\theta)-4(2-s_c)}.
\end{gather*}
The condition $\theta=0^+$ gives 
$$ (a,r)\in \Gamma_{s_c},\quad (d,r)\in \Gamma_{-s_c}\quad\mbox{and}\quad (2p-1-\theta)d'=a.$$
$\bullet$ The linear term. Since $(a,\frac{Nr}{N+rs_c})\in\Gamma$, by Strichartz estimate and Sobolev injections, one has
\begin{eqnarray*}
\|e^{i\cdot\Delta^2}u_0\|_{L^a((T,\infty),L^r)}
&\lesssim&\||\nabla|^{s_c}e^{i\cdot\Delta^2}u_0\|_{L^a((T,\infty),L^{\frac{Nr}{N+rs_c}})}\lesssim\|u_0\|_{H^2}.
\end{eqnarray*}
$\bullet$ The term $F_2$. Using Hardy-Littlewood-Sobolev and H\"older inequalities, via the fact that $0<\psi_R<1$, one has 
\begin{eqnarray}
\|\psi_R\mathcal N\|_{r'}
&=&\|(I_\alpha*|\cdot|^b|u|^p)|\cdot|^b|u|^{p-2}\psi_Ru\|_{r'}\nonumber\\
&\lesssim&\|u\|_r^{2p-1-\theta}\|\psi_Ru\|_{r_1}^\theta,\label{use}
\end{eqnarray}
where
$$1+\frac{2b+\alpha}N=\frac{2p-\theta}r+\frac\theta{r_1}.$$
Thus,
\begin{eqnarray*}
N+\alpha+2b
&=&\frac{N(2p-\theta)}r+\frac{N\theta}{r_1}\\
&=&\frac{(N-2s_c)(2p-\theta)-4(2-s_c)}2+\frac{N\theta}{r_1}\\
&=&N+\alpha+2b+\frac{N\theta}{r_1}-\theta\frac{4+2b+\alpha}{2(p-1)}.
\end{eqnarray*}
Because $p_*<p<p^*$, one gets
$$2<r_1=\frac{2N(p-1)}{4+2b+\alpha}<\frac{2N}{N-4}.$$
So, denoting $I_2:=(T-\varepsilon^{-\beta},T)$, it follows that for $\lambda:=\lambda_N\in(0,1)$,
\begin{eqnarray*}
\|\psi_R\mathcal N\|_{L^{d'}(I_2,L^{r'})}
&\lesssim&\|u\|_{L^a(I_2,L^r)}^{2p-1-\theta}\|\psi_Ru\|_{r_1}^\theta\\
&\lesssim&|I_2|^{\frac{2p-1-\theta}a}\|u\|_{L^\infty(I_2,H^2)}^{2p-1-\theta}\|\psi_Ru\|_{r_1}^\theta\\
&\lesssim&|I_2|^{\frac{2p-1-\theta}a}\|\psi_Ru\|_{r_1}^\theta\\
&\lesssim&\varepsilon^{-\beta\frac{2p-1-\theta}a}\|\psi_Ru\|^{\theta\lambda}.
\end{eqnarray*}
Now, by the assumptions of the scattering criterion, one has
$$\int_{\R^N}\psi_R(x)|u(T,x)|^2\,dx<\epsilon^2.$$
Moreover, a computation with use of \eqref{S} and the properties of $\psi$ give
$$|\frac d{dt}\int_{\R^N}\psi_R(x)|u(t,x)|^2\,dx|\lesssim R^{-1}.$$
Then, for any $T-\varepsilon^{-\beta}\leq t\leq T$ and $R>\varepsilon^{-2-\beta}$, yields
$$\|\psi_Ru(t)\|\leq\Big( \int_{\R^N}\psi_R(x)|u(T,x)|^2\,dx+C\frac{T-t}R\Big)^\frac12\leq C\varepsilon.$$
Then, for small $\beta>0$, there exists $\eta>0$ such that
\begin{eqnarray*}
\|\psi_R\mathcal N\|_{L^{d'}(I_2,L^{r'})}
&\lesssim&\varepsilon^{-\beta\frac{2p-1-\theta}a}\|\psi_Ru\|^{\theta\lambda}\\
&\lesssim&\varepsilon^{\theta\lambda-\beta\frac{2p-1-\theta}a}\\
&\lesssim&\varepsilon^{\eta}.
\end{eqnarray*}
On the other hand, by Hardy-Littlewood-Sobolev inequality
\begin{eqnarray*}
\|(1-\psi_R)\mathcal N\|_{r'}
&=&\|(I_\alpha*|\cdot|^b|u|^p)|\cdot|^b|u|^{p-2}(1-\psi_R)u\|_{r'}\\
&\lesssim&\||x|^b\|_{L^{\mu_1}(|x|<1)}\|u\|_r^{2p-1-\theta}\|u\|_{\frac{2N}{N-4}}^\theta\||x|^b\|_{L^{\mu_2}(|x|>\frac R2)}\\
&+&\||x|^b\|_{L^{\mu}(|x|>1)}\|u\|_r^{2p-1-\theta}\|u\|^\theta\||x|^b\|_{L^{\mu}(|x|>\frac R2)}\\
&\lesssim&R^{N+b\mu_2}\|u\|_r^{2p-1-\theta}\|u\|_{r_1}^\theta+R^{N+b\mu}\|u\|_r^{2p-1-\theta}\|u\|^\theta,
\end{eqnarray*}
where
\begin{gather*}
N+b\mu<0;\\
N+b\mu_1>0;\\
N+b\mu_2<0;\\
1+\frac\alpha N=\frac{2p-\theta}r+\frac\theta2+\frac2{\mu};\\
1+\frac\alpha N=\frac{2p-\theta}r+\frac{\theta(N-4)}{2N}+\frac1{\mu_1}+\frac1{\mu_2}.
\end{gather*}
Compute
\begin{eqnarray*}
N+\alpha
&=&\frac{N(2p-\theta)}r+\frac{\theta N}{2}+\frac{2N}{\mu}\\
&=&\theta(s_c-\frac N2)-2(2-s_c)+2(p-1)(\frac N2-s_c)+2(\frac N2-s_c)+\frac{\theta N}{2}+\frac{2N}{\mu}\\
&=&\theta s_c-2(2-s_c)+4+2b+\alpha+2(\frac N2-s_c)+\frac{2N}{\mu}.
\end{eqnarray*}
Thus,
$$\mu=\frac{2N}{-2b-\theta s_c}>-\frac Nb.$$
Moreover,
\begin{eqnarray*}
N+\alpha
&=&\frac{N(2p-\theta)}r+\frac{\theta(N-4)}{2}+N(\frac1{\mu_1}+\frac1{\mu_2})\\
&=&\theta(s_c-\frac N2)-2(2-s_c)+2(p-1)(\frac N2-s_c)+2(\frac N2-s_c)+\theta(\frac{N}{2}-2)+N(\frac1{\mu_1}+\frac1{\mu_2})\\
&=&\theta(s_c-2)-2(2-s_c)+4+2b+\alpha+2(\frac N2-s_c)+N(\frac1{\mu_1}+\frac1{\mu_2})\\
&=&\theta(s_c-2)+2b+\alpha+N+N(\frac1{\mu_1}+\frac1{\mu_2}).
\end{eqnarray*}
Taking $\epsilon<\theta(2-s_c)$ and $\frac N{\mu_1}=-b+\epsilon$, one gets 
$$
\frac N{\mu_2}
=-b-\epsilon-\theta(s_c-2)>-b.
$$
Then,
\begin{eqnarray*}
\|(1-\psi_R)\mathcal N\|_{L^{d'}(I_2,L^{r'})}
&\lesssim&(R^{N+b\mu_2}+R^{N+b\mu})\|u\|_{L^a(I_2,L^r)}^{2p-1-\theta}\\
&\lesssim&(R^{N+b\mu_2}+R^{N+b\mu})|I_2|^{\frac{2p-1-\theta}{a}}\|u\|_{L^\infty(I_2,H^2)}^{2p-1-\theta}\\
&\lesssim&(R^{N+b\mu_2}+R^{N+b\mu})|I_2|^{\frac{2p-1-\theta}{a}}\\
&\lesssim&(R^{N+b\mu_2}+R^{N+b\mu})\epsilon^{-\beta\frac{2p-1-\theta}{a}}.
\end{eqnarray*}
Regrouping the above estimates and choosing $R>\max\{[\epsilon^{1+\beta\frac{2p-1-\theta}{a}}]^\frac{-1}{N+b\mu_2},[\epsilon^{1+\beta\frac{2p-1-\theta}{a}}]^\frac{-1}{N+b\mu}\}$, one gets for some $\lambda>0$,
\begin{eqnarray*}
\|F_2\|_{L^a((T,\infty),L^r)}
&\lesssim&\varepsilon^\lambda.
\end{eqnarray*}
$\bullet$ The term $F_1$. Following lines in \cite{lcg} via the estimate \eqref{use}, which gives $\|\mathcal N\|_{r'}\lesssim\|u\|_{H^2}^{2p-1}$, there is a positive real number denoted also by $\lambda>0$ such that one gets
\begin{eqnarray*}
\|F_1\|_{L^a((T,\infty),L^r)}
&\lesssim&\varepsilon^{\lambda}.
\end{eqnarray*}
The proof is closed via the three above points.
\end{proof}
Now, one proves the scattering criterion. Taking account of Duhamel formula, there exists $\mu>0$ such that
$$\|e^{i.\Delta^2}u(T)\|_{L^a((0,\infty),L^r)}=\|e^{i(.-T)\Delta^2}u(T)\|_{L^a((T,\infty),L^r)}\lesssim\epsilon^\mu.$$
So, with Lemma \ref{tch} via the absorption result Lemma \ref{abs}, one gets
$$\|u\|_{L^a((0,\infty),L^r)}<\infty.$$
With Lemma \ref{tch}, one gets for $u_+:=e^{-iT\Delta^2}u(T)+i\int_T^\infty e^{-is\Delta^2}\mathcal N\,ds$,
\begin{eqnarray*}
\|u(t)-e^{it\Delta^2}u_+\|_{H^2}
&=&\|\int_t^\infty e^{i(t-s)\Delta^2}\mathcal N\,ds\|_{H^2}\\
&\lesssim&\|u\|_{L^\infty((t,\infty),H^2)}^{\theta}\|u\|_{L^a((t,\infty),L^r)}^{2q-1-\theta}\\
&&\to0.
\end{eqnarray*}
This finishes the proof.
\end{proof}
\section{Scattering}
Theorem \ref{sctr} about the scattering of energy global solutions to the focusing problem \eqref{S} follows with Proposition \ref{crt} via Lemma \ref{evac1}.


\end{document}